\documentclass[11pt,reqno,letterpaper]{amsart}
\usepackage{color}
\usepackage[colorlinks=true, allcolors=blue,backref=page]{hyperref}
\usepackage{amsmath, amssymb, amsthm}
\usepackage{mathrsfs}
\usepackage{mathtools}
\usepackage[noabbrev,capitalize,nameinlink]{cleveref}
\crefname{equation}{}{}
\usepackage{fullpage}
\usepackage[noadjust]{cite}
\usepackage{graphics}
\usepackage{pifont}
\usepackage{tikz}
\usepackage{bbm}
\usepackage[T1]{fontenc}

\usetikzlibrary{arrows.meta}

\usepackage{environ}
\usepackage{framed}
\usepackage{url}
\usepackage[linesnumbered,ruled,vlined]{algorithm2e}
\usepackage[noend]{algpseudocode}
\usepackage[labelfont=bf]{caption}
\usepackage{cite}
\usepackage{framed}
\usepackage[framemethod=tikz]{mdframed}
\usepackage{appendix}
\usepackage{graphicx}
\usepackage[textsize=tiny]{todonotes}
\usepackage{tcolorbox}
\usepackage{enumerate}
\allowdisplaybreaks[1]
\usepackage{enumerate}
\usepackage{stmaryrd}
\usepackage[margin=1in]{geometry}

\usepackage[shortlabels]{enumitem}
\crefformat{enumi}{#2#1#3}
\crefrangeformat{enumi}{#3#1#4 to~#5#2#6}
\crefmultiformat{enumi}{#2#1#3}
{ and~#2#1#3}{, #2#1#3}{ and~#2#1#3}

\DeclareSymbolFont{symbolsC}{U}{pxsyc}{m}{n}
\SetSymbolFont{symbolsC}{bold}{U}{pxsyc}{bx}{n}
\DeclareFontSubstitution{U}{pxsyc}{m}{n}
\DeclareMathSymbol{\medcircle}{\mathbin}{symbolsC}{7}

\crefname{algocf}{Algorithm}{Algorithms}

\crefname{equation}{}{} 
\AtBeginEnvironment{appendices}{\crefalias{section}{appendix}} 

\usepackage[color,final]{showkeys} 

\colorlet{refkey}{orange!20}
\colorlet{labelkey}{blue!30}

\crefname{algocf}{Algorithm}{Algorithms}

\numberwithin{equation}{section}
\newtheorem{theorem}{Theorem}[section]
\newtheorem{proposition}[theorem]{Proposition}
\newtheorem{lemma}[theorem]{Lemma}
\newtheorem{claim}[theorem]{Claim}

\crefname{claim}{Claim}{Claims}

\newtheorem*{question*}{Question}

\theoremstyle{definition}
\newtheorem{definition}[theorem]{Definition}

\newtheorem*{definition*}{Definition}

\theoremstyle{remark}
\newtheorem{remark}[theorem]{Remark}

\newcommand{\snorm}[1]{\lVert#1\rVert}
\newcommand{\sang}[1]{\langle #1 \rangle}

\newcommand{\mb}{\mathbb}

\newcommand{\mbm}{\mathbbm}
\newcommand{\mc}{\mathcal}

\newcommand{\mr}{\mathrm}

\newcommand{\on}{\operatorname}

\newcommand{\wt}{\widetilde}

\newcommand{\eps}{\varepsilon}

\let\originalleft\left
\let\originalright\right
\renewcommand{\left}{\mathopen{}\mathclose\bgroup\originalleft}
\renewcommand{\right}{\aftergroup\egroup\originalright}

\allowdisplaybreaks

\newif\ifpublic
\publictrue

\ifpublic

\newcommand{\ignore}[1]{}

\else

\fi

\title{Distribution of the threshold for the symmetric perceptron}

\author[A2]{Ashwin Sah}
\author[A3]{Mehtaab Sawhney}
\address{Department of Mathematics, Massachusetts Institute of Technology, Cambridge, MA 02139, USA}
\email{\{asah,msawhney\}@mit.edu}

\thanks{Sah and Sawhney were supported by NSF Graduate Research Fellowship Program DGE-1745302. Sah was supported by the PD Soros Fellowship. Sawhney was supported by the Churchill Foundation.}

\begin{document}
\maketitle
\begin{abstract}
We derive an explicit distribution for the threshold sequence of the symmetric binary perceptron with Gaussian disorder, proving that the critical window is of constant width.
\end{abstract}

\section{Introduction}\label{sec:introduction}
We fix a real number $\kappa>0$ throughout the rest of this paper.
\begin{definition}\label{def:model}
Let $\Sigma_n = \{\pm 1\}^n$. Given an infinite sequence $M=(M_j)_{j\ge 1}$ of vectors in $\mb{R}^n$, let
\[S_j(M):=\{x\in\Sigma_n\colon|\sang{x,M_j}|\le\kappa\sqrt{n}\},\qquad S^m(M):=\bigcap_{j=1}^mS_j(M).\]
Furthermore, define the \emph{threshold} $\tau=\tau(M)$ as the first index such that
\[S^{\tau-1}(M)\neq\emptyset\text{ ~and }S^\tau(M)=\emptyset,\]
or $+\infty$ if it does not exist.
Let $G=(G_j)_{j\ge 1}$ be a sequence of independent standard $n$-dimensional Gaussian vectors, i.e., $G_j\sim\mc{N}(0,I_n)$. The \emph{threshold for the symmetric binary perceptron with Gaussian disorder} is $\tau=\tau(G)$ (note $\tau(G)<+\infty$ almost surely).
\end{definition}

We note that the threshold as defined in \cref{def:model} is up to rescaling equivalent to the \emph{storage capacity} of the binary perceptron model. We have chosen to work with Gaussian disorder as it marginally simplifies certain computations. The limiting distribution can immediately be extended to Rademacher disorder, although various quantitative error terms will differ and the mean and variance of an associated log-normal distribution $Z^\ast$ must be slightly changed.

We will require the following basic special functions.

\begin{definition}\label{def:gaussian-special}
Let $Z,Z_1,Z_2\sim\mc{N}(0,1)$ be independent standard Gaussians. Let $p = \mb{P}[|Z|\le \kappa]$ and define the critical value 
\[\alpha_c:= - \log 2/\log p.\]
We also define
\[\mu_2 = \frac{\mb{E}[Z^2\mbm{1}_{|Z|\le\kappa}]}{p},\quad\beta = -\frac{\sqrt{\alpha_c}}{2}(1-\mu_2).\]
Finally we define the \emph{pair probability function}
\[q(\gamma) = \mb{P}\big[|\sqrt{\gamma}Z_1+\sqrt{1-\gamma}Z_2|\le \kappa\wedge|\sqrt{\gamma}Z_1-\sqrt{1-\gamma}Z_2|\le\kappa\big].\]
\end{definition}

We note that $-1/2<\beta<0$ follows from \cite[(3.8)]{Alt22}. We now state our main theorem. 
\begin{theorem}\label{thm:main}
Let $G$ and $\tau=\tau(G)$ be as in \cref{def:model}. There exists $\theta > 0$ such that the following holds for all sufficiently large $n$. Let $Z^\ast\sim\mc{N}(\frac{1}{4}\log(1-4\beta^2), -\frac{1}{2}\log(1-4\beta^2))$ and choose $k$ with $k + \alpha_cn\in\mb{Z}$. We have
\[\bigg|\mb{P}[\tau\le k + \alpha_cn] - \mb{E}\bigg[\exp\bigg(\frac{-e^{Z^{\ast}}p^k}{2}\bigg)\bigg]\bigg|\le n^{-\theta}.\]
\end{theorem}
\begin{remark}\label{rmk:Lp}
The polynomial rate specified by \cref{thm:main} is essentially the limit of our method. Combining the statement of \cref{thm:main} with the tails bounds given by Altschuler \cite[Theorem~2]{Alt22}, one can prove that there exists $C$ such that for any fixed $q\ge 1$ and $n$ large in terms of $q$:
\[\mb{E}|\tau-\alpha_cn|^q\le C^q.\]
We also note that \cref{thm:main} disproves \cite[Conjecture~1]{Alt22} since the upper tail $\mb{P}[\tau\ge k+\alpha_cn]$ is of exponential type in $k>0$.
\end{remark}

\subsection{Previous work}
A toy model for neural networks, the (asymmetric) perceptron model has been a model of substantial interest, in part stemming from conjectures on the asymptotic threshold due to Krauth and Mezard \cite{KM89}. Despite progress on a number of related problems, progress on the perceptron has proved considerably more difficult. In this direction, groundbreaking work of Ding and Sun \cite{DS19} establishes the conjectural lower bound for the threshold and recent work of Xu \cite{Xu21} has established an analog of Friedgut's sharp threshold theorem \cite{Fri99} for such models (see work of Nakajima and Sun \cite{NS22} for an alternate proof). In particular these sharp threshold results establish that for general perceptron models, $\tau$ typically lives within an interval of length $o(n)$.

The symmetric perceptron model, introduced in work of Aubin, Perkins, and Zdeborov\'{a} \cite{APZ19}, has proven to be considerably more tractable at a mathematical level due the availability of the moment method applied directly to the number of solutions (while for the asymmetric perceptron model a substantially more involved conditioning scheme is required). It provides a test-bed for various conjectural phenomena since the solution space geometries of symmetric and asymmetric perceptron are expected to behave similarly. In particular, work of Aubin, Perkins, and Zdeborov\'{a} \cite{APZ19} established that w.h.p.~$\tau < (\alpha_c+\epsilon)n$ and with positive probability that $\tau > (\alpha_c-\epsilon)n$.
The determination of the threshold (with scaling window $o(n)$) at $\alpha_cn$ with high probability was proven independently by Abbe, Li, and Sly \cite{ALS21} and Perkins and Xu \cite{PX21}\footnote{We note that the work of Aubin, Perkins, and Zdeborov\'{a} \cite{APZ19} and Perkins and Xu \cite{PX21} was conditional on a numerical hypothesis which was verified in the work of  Abbe, Li, and Sly \cite{ALS21}.}. Both the works of Abbe, Li, and Sly \cite{ALS21} and Perkins and Xu \cite{PX21} establish that $\log |S^m(G)|$ concentrates around $\log\mb{E}[|S^m(G)|]$. Furthermore, the work of Abbe, Li, and Sly \cite{ALS21} establishes that $S^{m}(G)$ has a log-normal distribution (when $m$ is linearly bounded away from the threshold) via a modification of ``small subgraph conditioning'' made suitable for dense models. Finally, in recent work Altschuler \cite{Alt22} substantially improved the concentration window by proving that the first and second moment of number of solutions at $\lceil \alpha_c n\rceil$ match up to a constant. Adapting the argument for concentration given by Perkins and Xu \cite{PX21}, Altschuler derived that the window of concentration is $O(\log n)$ w.h.p. We refer the reader to \cite[Section~2]{Alt22} for a more extensive discussion of previous works on the threshold, discussion of threshold for more general constraint satisfaction problems, and related work.

We also note that there has recently been work on understanding efficient algorithms for the symmetric perceptron; in particular work of Abbe, Li, and Sly \cite{ALS22} proves that efficient algorithms can find exponentially rare clusters at sufficiently low density and work of Gamarnik, K{\i}z{\i}lda{\u{g}}, Perkins, and Xu \cite{GKPX22} examines the limits of efficient algorithms for the symmetric perceptron through the lens of the multi-Overlap Gap Property.

\subsection{Stopping time and Poisson sampling heuristic}\label{sub:outline}
The key ingredient our in work is an early stopping time argument, which combined with an understanding of the geometry of the remaining solutions is sufficient to derive a distribution for the threshold. More precisely, define the (non-random) cutoff time $\tau_\mr{pre} = \lfloor \alpha_c n - \eta\log n\rfloor$ for some sufficiently small constant $\eta>0$. By mimicking the techniques of Abbe, Li, and Sly \cite{ALS21} (for Gaussian disorder), one can prove that the number of solutions to the perceptron at $\tau_\mr{pre}$ follows a log-normal distribution; we note here that the fact that the second and first moments match so close to the threshold is a critical insight found in the work of Altschuler \cite{Alt22}. Furthermore, the moment computations of Altschuler (in particular the ``frozen'' nature of typical solutions) can be used to prove that all pairs of solutions are essentially orthogonal (e.g.~have dot product bounded in magnitude by $n^{1/2}\log n$).

The crucial idea at this point is to note that since w.h.p.~at $\tau_\mr{pre}$ there are a small polynomial number of solutions, the pairwise dot product of solutions being near zero can be translated to near-independence of the event of whether these solutions will survive the next slices. In particular, the total variation distance between the true distribution and a natural product distribution is polynomially small. The desired result then follows noting that there are only logarithmically more steps and thus the computation of the threshold sequence reduces to a computation with product distributions: we have some set $S$ of solutions at time $\tau_\mr{pre}$ and each solution is retained for $t-\tau_\mr{pre}$ extra slices with probability $p^{t-\tau_\mr{pre}}$, so an independence heuristic suggests that $\mb{P}[\mr{Bin}(|S|,p^{t-\tau_\mr{pre}})=0]$ is roughly the chance that we empty the solution set in $t$ steps, where $\mr{Bin}$ denotes a binomial random variable. In our regime, this can be approximated via a Poissonian heuristic that gives rise to the expression in \cref{thm:main} (after incorporating the fact that $|S|/\mb{E}|S|$ satisfies a log-normal distribution). We note that the strange factor of $2$ arising in the distribution in \cref{thm:main} stems from the fact that $v\in S$ if and only if $-v\in S$.

We note that in implementing the above sketch in order to give polynomial error rates, the primary difficulty is quantifying the work of Abbe, Li, and Sly \cite{ALS21} (and adjusting it to both handle when the number of hyperplanes is near the threshold and to allow for Gaussian disorder). In order to provide a quantitative rate we transform the necessary cycle count distributional claims under conditional models to distributional questions about unconditioned Gaussian polynomials (via orthogonal invariance of Gaussians and various 
expectation and variance computations) and then provide quantitative rates via Stein's method (in paricular through the use of exchangeable pairs), rather than the method of moments used in \cite{ALS21}.
We note one curious feature is that the distribution of solutions at $\tau_\mr{pre}$, while still being log-normal, does not have the same mean and variance parameters as the Rademacher case derived in \cite{ALS21}. The difference ultimately arises from the fact that if $M\sim\mc{N}(0,1)$, then $M^2$ has nontrivial variance whereas $M^2$ is constant if $M\sim\mr{Unif}(\{\pm 1\})$. (This difference manifests in the need to track the degenerate $2$-cycle count corresponding to $k=1$ in \cref{def:general-cycle-count}; we note that the $k=1$ expression in the Rademacher case would be deterministic, hence the difference.)

We note that the methods of this paper appear to more generally applicable; in particular, our approach can likely be adapted give a precise characterization of the threshold for solutions in constrained $k$-XORSAT which was shown to have an $O(1)$-size scaling window by Pittel and Sorkin \cite{PS16} (see also \cite{DM02,DGMMPR10}). We intend to return to this subject in future work.

\subsection*{Organization}
The remainder of the paper is organized as follows. In \cref{sec:proof-main}, we prove \cref{thm:main} conditional on the ``strong freezing'' of the solution space and on the log-normality of the number of solutions (\cref{prop:moment-output}). This argument implements the main thrust of the paper outlined in \cref{sub:outline}. We reduce \cref{prop:moment-output} to a concrete moment computation (\cref{lem:second-moment}) in \cref{sub:moment-reduction}. In \cref{sec:moment-prelim}, we state preliminaries used to verify the moment computation including a suitable quantification of cycle count convergence as defined in the work of Abbe, Li, and Sly \cite{ALS21}. Finally, in \cref{sec:moment}, we prove the necessary moment estimates to imply \cref{prop:moment-output}, which in turn closely follows the work of Altschuler \cite{Alt22} and Abbe, Li, and Sly \cite{ALS21}.

\subsection*{Acknowledgements}
We thank Will Perkins for bringing this problem to our attention and pointing out the work of Pittel and Sorkin \cite{PS16}, and for various useful comments on the draft. We also thank Mark Sellke for finding a number of typos and Michael Ren for useful conversations.

\subsection*{Notation}
All dependences of notation or asymptotics on $\kappa$ have been suppressed, and we treat it as a fixed absolute constant throughout the paper. We write $f=O(g)$ to mean that $f\le Cg$ for some absolute constant $C$, and $g=\Omega(f)$ and $f\lesssim g$ to mean the same. We write $f=o(g)$ if for all $c > 0$ we have $f\le cg$ once the implicit growing parameter (typically $n$) grows large enough, and $g=\omega(f)$ means the same. Furthermore throughout this paper all logarithms are base $e$.

For positive semidefinite $\Sigma\in\mb{R}^{d\times d}$ we let $\mc{N}(0,\Sigma)$ be the Gaussian vector with covariance matrix $\Sigma$. A common choice is $\Sigma=I_d$, the standard Gaussian vector with identity covariance. For a discrete random variable $X$ with nonzero probabilities $p_1,\ldots,p_m$ for its atoms, its entropy is
\[H(X)=-\sum_{i=1}^mp_i\log p_i.\]
The Kullback--Leibler divergence between $X,Y$ defined on the same atom set with probabilities $q_i,r_i$ respectively is
\[\on{KL}(X\parallel Y)=\sum_{i=1}^mq_i\log\frac{q_i}{r_i}\]
and we define the real function $H(x) = H(\mr{Ber}(x)) = -x\log x-(1-x)\log(1-x)$. If $X,Y$ are supported on $\mb{R}^d$ with probability densities $q,r$ their KL divergence is
\[\on{KL}(q\parallel r)=\int_{\mb{R}^d}q(x)\log\frac{q(x)}{r(x)}dx.\]
For real matrices $A$ and $B$, let
\[\langle A, B\rangle= \on{tr}(A^TB),\qquad \|A\|_{\mr{HS}} = \sqrt{\on{tr}(A^TA)} = \sqrt{\langle A, A\rangle}.\]
Furthermore for an order $k$ tensor define
\[\|A\|_{\mr{op}} = \sup_{|v_i|=1}|A(v_1,v_2,\ldots,v_k)|.\] Finally for $f\in\mc{C}^k(\mathbb{R}^n)$ define the $k^{th}$ derivative (tensor) operators evaluated at $u_1,\ldots,u_k\in\mb{R}^n$ as 
\[\langle D^kf(x),(u_1,\ldots,u_k)\rangle = \sum_{i_1,i_2,\ldots,i_k\in [n]}\frac{\partial^k f}{\partial x_{i_1}\ldots \partial x_{i_k}}(u_1)_{i_1}\ldots(u_k)_{i_k}\]
and define
\[M_r(g) = \sup_{x\in\mathbb{R}^n} \|D^rg(x)\|_\mr{op}.\]
For a matrix $A$ define $\snorm{A}_{\mr{HS}}^2=\sum_{i,j}|a_{ij}|^2$.

\section{Proof of \texorpdfstring{\cref{thm:main}}{Theorem 1.3}}\label{sec:proof-main}
In order to derive our main result we will assume the following pair of structural properties of the solution space and its size. Here and beyond, we fix some $\eta>0$ that will be chosen small in terms of only $\kappa$ at the end of the argument. For now, we leave it unspecified in the various statements that appear, merely using that it is sufficiently small in the proofs.
\begin{proposition}\label{prop:moment-output}
There exists $\gamma>0$ such that the following holds. Let $\tau_\mr{pre}=\lfloor\alpha_cn-\eta\log n\rfloor$. For $n$ sufficiently large, we have the following pair of estimates on the size of the solution space $X(G):=|S_{\tau_\mr{pre}}(G)|$ for $n$ sufficiently large.
\begin{enumerate}[1.]
\item With probability at least $1-n^{-1}$ there do not exist $x_1,x_2\in S_{\tau_\mr{pre}}(G)$ such that $|\sang{x_1,x_2}|\in [n^{1/2}\log n, n-1]$.
\item Let $Z^\ast = \mc{N}(\frac{1}{4}\log(1-4\beta^2), -\frac{1}{2}\log(1-4\beta^2))$. Then for all $t\in\mb{R}$, 
\[\mb{P}\bigg[\frac{X(G)}{\mb{E}X(G)}\ge e^t\bigg] = \mb{P}[Z^\ast\ge t] \pm n^{-\gamma}.\]
\end{enumerate}
\end{proposition}

We will also require the following tail bounds on the number of solutions. The second appears explicitly as \cite[Theorem~3]{Alt22} which in turns builds closely on work of Perkins and Xu \cite{PX21}.

\begin{proposition}[{\cite[Theorem~3]{Alt22}}]\label{prop:tails}
We have the following.
\begin{enumerate}[1.]
\item For all $t\ge 0$ we have
\[\mb{P}[|S_t(G)|\ge 1]\le 2^np^t.\]
\item There exists a constant $c=c_{\ref{prop:tails}}>0$ such that following holds. Fix $H\ge c^{-1}$. Suppose that $n$ is sufficiently large in terms of $H$. Then
\[\mb{P}[S_{\lfloor\alpha_cn-H\log n\rfloor}(G) = \emptyset]\le\exp(-cH\log n).\]
\end{enumerate}
\end{proposition}

We require the following basic bound on the total variation distance between a pair of normal distributions.
\begin{lemma}\label{lem:KL}
For positive definite symmetric matrices $\Sigma_1,\Sigma_2\in\mb{R}^{d\times d}$, we have 
\[2\on{TV}(\mc{N}(0,\Sigma_1),\mc{N}(0,\Sigma_2))^2\le \on{KL}(\mc{N}(0,\Sigma_1)\parallel\mc{N}(0,\Sigma_2)) = \frac{1}{2}\big(\log(\det(\Sigma_2\Sigma_1^{-1}))-d+\on{tr}(\Sigma_2^{-1}\Sigma_1)\big).\]
\end{lemma}
\begin{proof}
The first inequality is by Pinsker's inequality (see e.g.~\cite[Lemma~2]{Can22}), holding for general distributions. The equality is a straightforward Gaussian integral (see e.g.~\cite[Page~13]{Duc}).
\end{proof}

Now we prove \cref{thm:main} given these inputs.
\begin{proof}[Proof of \cref{thm:main}]
Let $\tau_{\mr{pre}}$, $\eta$ and $X(G)$ be as in \cref{prop:moment-output}. Let 
\[\mc{E}_1 = \{X(G)\ge \mb{E}[X(G)]\cdot \exp((\log n)^{3/4})\} \cup \{X(G)\le \mb{E}[X(G)]\cdot\exp(-(\log n)^{3/4})\}\]
and 
\[\mc{E}_2 = \{|\sang{x_1,x_2}|\in [n^{1/2}\log n,n-1]:x_1,x_2\in S_{\tau_\mr{pre}}\}.\]
We now reveal set $S = S_{\tau_{\mr{pre}}}(G) \cap \{x:x_1 = 1\}$ and note that $\mc{E}_1$, $\mc{E}_2$, and $X(G)$ are deterministic given $S_{\tau_{\mr{pre}}}(G)$. We note that \cref{prop:moment-output} implies that $\mb{P}[\mc{E}_1 \cup \mc{E}_2] \le n^{-\Omega(1)}$ (since $|Z^\ast|\le(\log n)^{3/4}$ occurs with super-polynomially good probability). We condition on a revelation of $S$ such that $\mc{E}_1^c\cap\mc{E}_2^c$ holds. 

We now consider $G_t$ for $t\in(\tau_{\mr{pre}}, \tau_{\mr{pre}} + 2\eta\log n]$ and consider the distribution induced by 
\[\{\sang{G_t,x_i}/\sqrt{n}: x_i\in S, t\in(\tau_{\mr{pre}}, \tau_{\mr{pre}} + 2\eta\log n]\}.\]
Note that $G_t$ are independent for $t\in(\tau_{\mr{pre}}, \tau_{\mr{pre}} + 2\eta\log n]$. Let $M$ denote a $X(G)/2$ by $n$ matrix where each row corresponds to an entry of $S$. Notice that for each $t\in(\tau_{\mr{pre}}, \tau_{\mr{pre}} + \lceil 2\eta\log n\rceil]$ we have that
\[\{\sang{G_t,x_i}/\sqrt{n}: x_i\in S\} \sim \mc{N}(0, MM^T/n).\]
The crucial point is that since $MM^{T}/n$ is sufficiently diagonally dominant we will be able to prove that the distribution is close in total variation distance to $\mc{N}(0,I_{|S|})$. To prove this, note that 
\begin{align*}
\snorm{MM^{T}/n-I_{|S|}}_{\mr{op}}^2&\le \snorm{MM^{T}/n-I_{|S|}}_{\on{F}}^2\le |S|^2\max_{\substack{x_1,x_2\in S\\x_1\neq x_2}}|\sang{x_1,x_2}|^2/n^2\\
&\le |S|^2\cdot n^{-3/4}\le n^{-1/2}
\end{align*}
due to $\mc{E}_2^c$ holding and the fact that $S$ does not simultaneously contain any pair $\{x,-x\}$.

This immediately implies that all the eigenvalues of $MM^{T}/n$ are of the form $1\pm n^{-1/4}$ and therefore by \cref{lem:KL} we have that 
\begin{align*}
\on{TV}(\mc{N}(0,I_{|S|}),\mc{N}(0, MM^T/n))&\le (\on{Tr}(MM^T/n)-|S| + \log(\det(MM^T/n)))/2\\
&\le |S|n^{-1/4} + |S|\log(1+n^{-1/4})\le n^{-1/5}.
\end{align*}
The last inequality comes from $|S|=X(G)/2\le\mb{E}X(G)\cdot\exp((\log n)^{3/4})\le n^{O(\eta)}$ by the choice of $\tau_\mr{pre}$, and assuming $\eta$ is small enough.

As we are restricting attention to $t\in(\tau_{\mr{pre}},\tau_{\mr{pre}} + \lceil 2\eta\log n\rceil]$ we have that 
\[\on{TV}((\{\sang{G_t,x_i}/\sqrt{n}: x_i\in S\})_{t\in(\tau_{\mr{pre}}, \tau_{\mr{pre}} + \lceil 2\eta\log n\rceil]},\mc{N}(0,I_{|S|})^{\otimes\lceil 2\eta\log n\rceil}))\le n^{-1/6}.\]

Since $S$ satisfies $\mc{E}_1^c\cup\mc{E}_2^c$, this implies that for $t\in(\tau_{\mr{pre}}, \tau_{\mr{pre}} + \lceil 2\eta\log n\rceil]$ we have
\[\mb{P}[S_t = \emptyset| S_{\tau_{\mr{pre}}}] = (1-p^{t-\tau_{\mr{pre}}})^{X(G)/2} \pm n^{-1/6}.\]

We now choose a constant $\rho>0$ which is sufficiently small with respect to $\eta$. We first isolate the case where $t\in [\lceil\alpha_cn - \rho \log n\rceil, \lceil\alpha_cn + \rho \log n\rceil]$. Notice that we therefore have that 
\begin{align*}
\mb{P}[S_t = \emptyset] &= \mb{P}[S_t = \emptyset \cap (\mc{E}_1^c\cap \mc{E}_2^c)] \pm \mb{P}[\mc{E}_1 \cup \mc{E}_2]\\
&= \mb{P}[S_t = \emptyset \cap (\mc{E}_1^c \cap \mc{E}_2^c)] \pm n^{-\Omega(1)}\\
&= \mb{E}[(1-p^{t-\tau_{\mr{pre}}})^{X(G)/2}\mbm{1}_{\mc{E}_1^c\cap \mc{E}_2^c}] \pm n^{-\Omega(1)}\\
&= \mb{E}\bigg[\exp\bigg(\frac{-p^{t-\tau_{\mr{pre}}}\cdot X(G)}{2}\bigg)\mbm{1}_{\mc{E}_1^c\cap \mc{E}_2^c}\bigg] \pm n^{-\Omega(1)}.
\end{align*}
The final equality requires justification: note that $(1-p)^{t-\tau_{\mr{pre}}})^{X(G)}\le 1$, $1-x=e^{-x+O(x^2)}$ for all $|x|\le 1/2$, and
\[\exp(p^{2(t-\tau_{\mr{pre}})}\cdot X(G))\le 1+n^{-\Omega(1)}\]
by using that $\mc{E}_1^c$ holds and some basic computation.

Let $Z^\ast$ be as in \cref{prop:moment-output}. This implies that
\begin{align*}
\mb{E}\bigg[\exp\bigg(\frac{-p^{t-\tau_{\mr{pre}}}\cdot X(G)}{2}\bigg)\mbm{1}_{\mc{E}_1^c\cap \mc{E}_2^c}\bigg]&= \int_{0}^{1}\mb{P}\bigg[\exp\bigg(\frac{-p^{t-\tau_{\mr{pre}}}\cdot X(G)}{2}\bigg)\mbm{1}_{\mc{E}_1^c\cap \mc{E}_2^c}\ge u\bigg]du\\
&= \int_{0}^{1}\mb{P}\bigg[\exp\bigg(\frac{-e^{Z^{\ast}}\mb{E}[X(G)]p^{t-\tau_{\mr{pre}}}}{2}\bigg)\ge u\bigg]du\pm n^{-\Omega(1)}\\
&= \mb{E}\bigg[\exp\bigg(\frac{-e^{Z^{\ast}}\mb{E}[X(G)]p^{t-\tau_{\mr{pre}}}}{2}\bigg)\bigg]\pm n^{-\Omega(1)}\\
&= \mb{E}[\exp(-e^{Z^{\ast}}2^{n-1}p^t)]\pm n^{-\Omega(1)}.
\end{align*}
This implies the desired result for $t$ in the specified interval.

For $t\le \alpha_cn - \gamma \log n$, we have that 
\[\mb{P}[|S_t| = \emptyset] \le \mb{P}[|S_{\lfloor\alpha_cn -  \gamma \log n\rfloor}| = \emptyset] \le n^{-\Omega(1)},\]
using what we have established for $t=\lfloor\alpha_cn -  \gamma \log n\rfloor$ and explicitly computing $\mb{E}\exp(e^{-Z^\ast}2^{n-1}p^t)$, which implies the desired result in this range. Finally for $t\ge \alpha_cn + \gamma \log n$, we have that 
\[\mb{P}[|S_t| \neq \emptyset]\le\mb{E}[|S_t|]\le p^{(\gamma \log n)/2}\le n^{-\Omega(1)}\]
from \cref{prop:tails}(1) which immediately implies the desired result in this range. This finishes the proof.
\end{proof}

We also prove the remark following \cref{thm:main}.
\begin{proof}[{Proof of \cref{rmk:Lp}}]
Fix $q\ge 1$. Let $\tau^{\ast}$ be the random variable such that $\mb{P}[\tau^{\ast}\le k + \alpha_cn] = \mb{E}\big[\exp\big(-\frac{e^{Z^{\ast}}p^k}{2}\big)\big]$ for 
all choices of $k+\alpha_cn\in\mb{Z}$. Taking $H = 2qc^{-1}$ and applying \cref{prop:tails}(1,2), we have that 
\begin{align*}
\mb{E}|\tau - \alpha_cn|^q&\le \mb{E}[|\tau-\alpha_cn|^q\mbm{1}_{|\tau - \alpha_cn|\le H\log n}] + \mb{E}[|\tau - \alpha_cn|^q\mbm{1}_{|\tau - \alpha_cn|\ge H\log n}]\\
&\le \mb{E}[|\tau - \alpha_cn|^q\mbm{1}_{|\tau - \alpha_cn|\le H\log n}] + (2\alpha_cn)^q\exp(-2q\log n) + \mb{E}[\tau^q\mbm{1}_{\tau\ge 2\alpha_cn}]\\
&\le \mb{E}[|\tau - \alpha_cn|^q\mbm{1}_{|\tau - \alpha_cn|\le H\log n}] + n^{-q/2}\\
&\le \mb{E}[|\tau^{\ast} - \alpha_cn|^q\mbm{1}_{|\tau - \alpha_cn|\le H\log n}] + n^{-\Omega(1)}\\
&\le \mb{E}|\tau^\ast-\alpha_cn|^q + n^{-\Omega(1)}\\
&\le C^q
\end{align*}
Note that bounds derived on the upper tail for $\tau$ are an immediate consequence of \cref{prop:tails}(1) and the final inequality follows from noting that $\tau^\ast$ has subexponential tails by direct calculation.
\end{proof}

\subsection{Reduction to normalized second moment computation}\label{sub:moment-reduction}
We now formally state the second moment computation that we perform in the remainder of the paper and use these results to prove \cref{prop:moment-output}. We first define the key random variables used for small graph conditioning used in the definition of the moments.
\begin{definition}\label{def:cycle-count}
The \emph{$2k$-cycle count} of a sequence of vectors $M=(M_j)_{j\ge 1}$ is
\[C_k(m,M)=\bigg(\frac{1}{\sqrt{n}}\bigg)^k\bigg(\frac{1}{\sqrt{m}}\bigg)^k\sum_{\substack{i_1,\ldots,i_k\in[n]\text{ distinct}\\j_1,\ldots,j_k\in[m]\text{ distinct}}}\prod_{i=1}^kM_{j_\ell,i_\ell}M_{j_\ell,i_{\ell+1}}-\mbm{1}_{k=1}\sqrt{mn},\]
where we let $i_{k+1}=i_1$. Then the \emph{$K$-weighted cycle count} is
\[Y_K(m,M)=\sum_{k=1}^K\frac{2(2\beta)^kC_k(m,M)-(2\beta)^{2k}}{4k}.\]
\end{definition}
\begin{remark}
Notice that for $k=1$, this is a shifted sum of squares of Gaussians. Such an expression is trivial in the Rademacher setting and hence is not required in the work of Abbe, Li, and Sly \cite{ALS21}.
\end{remark}
We also define certain planted conditional distributions that are key in computing the moments.
\begin{definition}\label{def:planted}
For $t\in[-n,n]$ with $t\equiv n\pmod{2}$ define $v_t=(1,\ldots,1,-1,\ldots,-1)\in\mb{R}^n$ where there are $(n+t)/2$ many $1$s. Let $\Delta_0\sim\mc{N}(0,1)^{\otimes n}$. Let $\Delta_1$ be $x\sim\Delta_0$ conditional on the event $\{|\sang{x,v_n}|\le\kappa\sqrt{n}\}$. Let $\Delta_2(t)$ be $x\sim\Delta_1$ conditional on the further event $\{|\sang{x,v_t}|\le\kappa\sqrt{n}\}$. Then let $G^{(0)}=G$, which has its vectors drawn from $\Delta_0$, and let $G^{(1)}=(G^{(1)}_j)_{j\ge 1}$ have vectors drawn independently from $\Delta_1$ and $G^{(2)}(t)=(G^{(2)}_j(t))_{j\ge 1}$ have vectors drawn independently from $\Delta_2(t)$.
\end{definition}

We have the following weighted moment estimate.
\begin{lemma}\label{lem:second-moment}
There exists $\gamma>0$ such that the following holds. Let $m=\tau_\mr{pre}=\lfloor\alpha_cn-\eta\log n\rfloor$ and suppose $n$ is large. With $L=\eta\log n$, $K=\lceil\eta\log n\rceil$, and $\delta=n^{-\gamma}$ we have 
\begin{align}\label{eq:first-moment}
\mb{E}\bigg[\frac{X(G)}{\exp(Y_K(m,G)\mbm{1}[Y_K(m,G)\ge -L])}\bigg]\ge(1-\delta)\mb{E}X(G)\\
\mb{E}\bigg[\frac{X(G)^2}{\exp(2Y_K(m,G)\mbm{1}[Y_K(m,G)\ge -L])}\bigg]\le(1+\delta)\mb{E}X(G)^2.\label{eq:second-moment}
\end{align}
Additionally, we have
\begin{equation}\label{eq:pair-structure}
\sum_{t\in[-(n-1),-n^{1/2}\log n]\cup [n^{1/2}\log n,n-1]}2^n\binom{n}{(n+t)/2}\mb{P}[v_n,v_t\in S_{\tau_\mr{pre}}(G)]\le\exp(-(\log n)^{3/2}).
\end{equation}
\end{lemma}
To prove this moment estimate we will use Bayes' theorem, requiring an understanding of the $L$-weighted cycle count when the vectors of $M$ are drawn from the planted conditional distributions.
\begin{lemma}\label{lem:clt}
Let $m=\tau_\mr{pre}=\lfloor\alpha_cn-\eta\log n\rfloor$, $K=\lceil\eta\log n\rceil$, and $|t|\le\sqrt{n}\log n$. We define
\[V(M):=\bigg(\frac{C_1(m,M)}{\sqrt{2\cdot 1}},\ldots,\frac{C_K(m,M)}{\sqrt{2K}}\bigg),\quad\mu:=\bigg(\frac{(2\beta)^1}{\sqrt{2\cdot 1}},\ldots,\frac{(2\beta)^K}{\sqrt{2K}}\bigg).\]
For $g\in\mc{C}^3(\mb{R}^K)$ and $n$ sufficiently large we have 
\begin{align}
|\mb{E}[g(V(G^{(0)}))]-\mb{E}_{Z\sim\mc{N}(0,I_K))}[g(Z)]|&\le (M_1(g)+M_2(g)+M_3(g))n^{-1/5},\label{eq:conditional-0}\\
|\mb{E}[g(V(G^{(1)})-\mu)]-\mb{E}_{Z\sim\mc{N}(0,I_K))}[g(Z)]|&\le (M_1(g)+M_2(g)+M_3(g))n^{-1/5},\label{eq:conditional-1}\\
|\mb{E}[g(V(G^{(2)}(t))-2\mu)]-\mb{E}_{Z\sim\mc{N}(0,I_K))}[g(Z)]|&\le (M_1(g)+M_2(g)+M_3(g))n^{-1/5},\label{eq:conditional-2}
\end{align}
\end{lemma}

Now we prove \cref{prop:moment-output}.
\begin{proof}[Proof of \cref{prop:moment-output}]
We prove each of the items of \cref{prop:moment-output} in turn. For the first item notice by linearity of expectation that
\begin{align*}
\mb{E}[|\{x_1,x_2\in S_{\tau_{\on{pre}}}(G)\colon& |\sang{x_1,x_2}|\in [-(n-1),-n^{2/3}]\cup [n^{2/3},n-1]\}|] \\
&= \sum_{t\in [-(n-1),-n^{1/2}\log n]\cup [n^{1/2}\log n,n-1]}2^n\binom{n}{(n+t)/2}\mb{P}[v_n,v_t\in S_{\tau_\mr{pre}}(G)]\\
&\le \exp(-(\log n)^{3/2})
\end{align*}
by using \cref{eq:pair-structure}. The first item then follows immediately by Markov's inequality. 

For the second item, let $\wt{X}(G) = \frac{X(G)}{\exp(Y_K(m,G)\mbm{1}[Y_K(m,G)\ge-L])}$. By \cref{eq:first-moment,eq:second-moment} we have that 
\[\mb{E}\bigg[\bigg(\frac{\wt{X}(G)}{\mb{E}[X(G)]} - 1\bigg)^2\bigg]\le 3n^{-\gamma}.\]
Therefore with probability $1-n^{-\Omega(1)}$ we have that $\wt{X}(G) = (1\pm n^{-\Omega(1)})\mb{E}[X(G)]$. Equivalently this implies that with probability $1-n^{-\Omega(1)}$,
\[X(G)/\mb{E}[X(G)] = (1\pm n^{-\Omega(1)})\exp(Y_K(m,G)\mbm{1}[Y_K(m,G)\ge-L]).\]
The result then follows from \cref{lem:clt}: $Y_k(m,G)$ is a function of the coordinates of $V(G)$, which is close in test function distance to $\mc{N}(0,I_K)$. Therefore, the distribution of $X(G)/\mb{E}X(G)$ ought to look like an exponential of an appropriate Gaussian with mean $\sum_{k=1}^K\frac{-(2\beta)^{2k}}{4k}\approx(1/4)\log(1-4\beta^2)$ and variance $\sum_{k=1}^K\bigg(\frac{(2\beta)^k}{\sqrt{2k}}\bigg)^2\approx-(1/2)\log(1-4\beta^2)$ (these terms come from the linear expression of $Y_K(m,G)$ in terms of $V(G)$).

To make this quantitative, we choose appropriate bump functions $g$ in \cref{eq:conditional-0} and use the fact that Gaussian distributions are appropriately anticoncentrated; we omit the routine justification (see e.g.~\cite[Lemma~7.1]{BSS21}; see also \cref{sub:first-moment} for a similar computation). 
\end{proof}

\section{Preliminaries for moment computations}\label{sec:moment-prelim}
\subsection{Estimates on Gaussian pair probabilities}\label{sub:special-estimates}
We will require the following special function estimates on the pair point probabilities derived in the work of Abbe, Li, and Sly \cite{ALS21} and in work of Altschuler \cite{Alt22}. The first bullet is equivalent to the first bullet of \cite[Lemma~3.1]{Alt22}, the next three bullets appear exactly as in \cite[Lemma~3.1]{Alt22} (these were originally established in \cite[Section~4.7]{ALS21}), the fifth appears as \cite[(4.6)]{Alt22}, and the sixth appears as \cite[(4.3)]{Alt22}.
\begin{lemma}[From {\cite[Lemma~3.1]{Alt22},~\cite[(4.6)]{Alt22},~\cite[(4.3)]{Alt22}}]\label{lem:computation}
There exists $\eps=\eps_{\ref{lem:computation}} > 0$ such that the following hold. Let $F(\gamma) := H(\gamma) + \alpha_c\log q(\gamma)$.
\begin{enumerate}[{\bfseries{S\arabic{enumi}}}]
    \item\label{S1} $q''(1/2) = \frac{8\kappa^2e^{-\kappa^2}}{\pi}$
    \item\label{S2} $F''(1/2) \le-\eps$.
    \item\label{S3} $F(\gamma)$ is decreasing for $\gamma\in[0,\eps]$.
    \item\label{S4} For any $0\le a\le b\le 1/2$,\[\max_{\beta\in [a,b]} F(\gamma) = \max\{F(a),F(b),F(1/2)-\eps\}.\]
    \item\label{S5} \[F(1/n)\le F(0) - \eps/\sqrt{n}.\]
    \item\label{S6} \[(1-\mu_2)p = \sqrt{\frac{2}{\pi}}\kappa e^{-\kappa^2/2}\]
\end{enumerate}
\end{lemma}

\subsection{Quantification of Gaussian convergence}\label{sub:gaussian-convergence}
In order to prove the required Gaussian convergence of the random variables $Y_K$, we will prove \cref{lem:clt} in this section, essentially following the approach of Abbe, Li, and Sly \cite[Lemma~3.2]{ALS21}. We note their argument is written for Rademacher disorder, and additionally in order to prove \cref{rmk:Lp}, we will need to adapt various arguments in order to prove a stronger convergence result. To do so, we proceed indirectly via a linear algebraic change using the Gaussian structure. Specifically, we will first show that in a sense, sampling the vectors of $M$ from the conditional distributions $\Delta_1$ or $\Delta_2(t)$ for $|t|$ small (i.e., setting $M=G^{(1)}$ or $M=G^{(2)}(t)$) is very similar to sampling from $\Delta_0=\mc{N}(0,I_n)$ (i.e., $M=G^{(0)}=G$) up to a mean shift. Then we will show the necessary multivariate central limit theorem for $\Delta_0$ using multidimensional Stein's method machinery adapted from Meckes \cite{Mec09}. As our argument essentially serves as a more effective version of the bounds given in \cite[Lemma~3.2]{ALS21} in the Gaussian case, we will be brief with certain basic calculations.

To reduce to $\Delta_0$, we introduce a slightly more general cycle count which allows us to plug in different conditional distributions.
\begin{definition}\label{def:general-cycle-count}
Given $M$ and $-n<t<n$ with $t\equiv n\pmod{2}$, we define
\begin{align*}
\sigma_{j,i}(t,M)&=\frac{1}{(n+t)/2}\sum_{i'=1}^{(n+t)/2}M_{j,i'},&\wt{M}_{j,i}(t)=M_{j,i}-\sigma_{j,i}(t,M),\qquad&\text{for }1\le i\le\frac{n+t}{2},\\
\sigma_{j,i}(t,M)&=\frac{1}{(n-t)/2}\sum_{i'=(n+t)/2+1}^nM_{j,i'},&\wt{M}_{j,i}(t)=M_{j,i}-\sigma_{j,i}(t,M),\qquad&\text{for }\frac{n+t}{2}<i\le n.
\end{align*}
If we are further given a sequence of $\mb{R}^n$ vectors $\sigma=(\sigma_j)_{j\ge 1}$, we define
\[\wt{C}_k(m,t,M,\sigma)=\frac{1}{(mn)^{k/2}}\sum_{\substack{i_1,\ldots,i_k\in[n]\text{ distinct}\\j_1,\ldots,j_k\in[m]\text{ distinct}}}\prod_{\ell=1}^k\bigg(\wt{M}_{j_\ell,i_\ell}(t) + \sigma_{j_\ell,i_\ell}\bigg)\bigg(\wt{M}_{j_\ell,i_{\ell+1}}(t) + \sigma_{j_\ell,i_{\ell+1}}\bigg)-\mbm{1}_{k=1}\sqrt{mn}.\]
\end{definition}

The key point of this definition is that
\[C_k(m,M)=\wt{C}_k(m,t,M,\sigma(t,M))\]
and the distributions of $\wt{M}(t)$ and $\sigma(t,M)$ are independent in all 
relevant cases (namely $M=G^{(0)},G^{(1)},G^{(2)}(t)$). Furthermore, across all relevant cases we have that the distribution of $\wt{M}(t)$ is identical, so the distributions of $C_k(m,M)$ and $\wt{C}_k(m,t,G,\sigma(t,M))$ are identical. Thus it suffices to understand the differences in the models via the parameters $\sigma$, which merely encode the average of the first $(n+t)/2$ and last $(n-t)/2$ elements of each vector. Therefore, we prove the following, which essentially shows the resulting statistics $V(M)$ are the same up to shifting of means.

\begin{lemma}\label{lem:gaussian-stability}
Let $m=\tau_\mr{pre}=\lfloor\alpha_cn-\eta \log n\rfloor$ and $K=\lceil\eta\log n\rceil$. Given $1\le k\le K$ and $|t|\le\sqrt{n}\log n$, we have
\begin{align}\label{eq:stability-1}\mb{E}\big[\wt{C}_k(m,t,G,\sigma(t,G^{(1)}))-C_k(m,G)\big]&=(1+\wt{O}(n^{-1}))(2\beta)^k,\\
\mb{E}\big[\wt{C}_k(m,t,G,\sigma(t,G^{(2)}(t)))-C_k(m,G)\big]&=(1+\wt{O}(n^{-1/4}))2\cdot(2\beta)^k,\label{eq:stability-2}\\
\on{Var}\big[\wt{C}_k(m,t,G,\sigma(t,G^{(1)}))-C_k(m,G)\big]&=O(n^{-5/6}\on{Var}[C_k(m,G)]),\label{eq:stability-var-1}\\
\on{Var}\big[\wt{C}_k(m,t,G,\sigma(t,G^{(2)}(t)))-C_k(m,G)\big]&=O(n^{-5/6}\on{Var}[C_k(m,G)]).\label{eq:stability-var-2}
\end{align}
\end{lemma}
\begin{proof}
We define
\[\rho_1^{(0)}=\sigma_{1,1}(t,G^{(0)}),\rho_2^{(0)}=\sigma_{1,n}(t,G^{(0)})\]
and similar for $G^{(1)},G^{(2)}(t)$. Note that $\sigma_{1,i}$ for all $i\in\{1,\ldots,(n+t)/2\}$ will give exactly $\rho_1$ and $\sigma_{1,i}$ for all $i\in\{(n+t)/2+1,\ldots,n\}$ will give exactly $\rho_2$. Furthermore, as we vary $j$ in $\sigma_{j,i}$, we do not have the exact same value but we have an independent value that is drawn from the same distribution. For instance, note that $\rho_1^{(0)}=(G_{1,1}+\cdots+G_{1,(n+t)/2})/((n+t)/2)$ and $\rho_2^{(0)}$ some Gaussians, while $(\rho_1^{(1)},\rho_2^{(1)})$ have the distribution of $(x,y)=(\rho_1^{(0)},\rho_2^{(0)})$ conditioned on $|(n+t)x/2+(n-t)y/2|\le\kappa\sqrt{n}$, and $(\rho_1^{(2)},\rho_2^{(2)})$ have the distribution of $(x,y)=(\rho_1^{(1)},\rho_2^{(1)})$ further conditioned on $|(n+t)x/2-(n-t)y/2|\le\kappa\sqrt{n}$.

We first note that
\begin{equation}\label{eq:rho-expectation}
\mb{E}\rho_1^{(b)}=\mb{E}\rho_2^{(b)}=0
\end{equation}
for all $b\in\{0,1,2\}$, and
\begin{equation}\label{eq:rho-covariance}
\mb{E}\rho_1^{(b)}\rho_2^{(b)}=0
\end{equation}
for $b\in\{0,2\}$. This can be seen from symmetry. Note that $b=1$ does have some covariance. Additionally, we have
\[\mb{E}(\rho_1^{(0)})^2=\frac{1}{(n+t)/2},\quad\mb{E}(\rho_2^{(0)})^2=\frac{1}{(n-t)/2}.\]
Finally, we can compute
\[\mb{E}(\rho_1^{(2)})^2=(1+\wt{O}(n^{-1/4}))\frac{2\mu_2}{n+t},\quad\mb{E}(\rho_2^{(2)})^2=(1+\wt{O}(n^{-1/4}))\frac{2\mu_2}{n-t},\]
using that the distribution of $(\rho_1^{(2)},\rho_2^{(2)})$ is that of $(\rho_1^{(0)},\rho_2^{(0)})$, which are independent Gaussians, conditional on two explicit linear inequalities. (The error term comes from the possible deviation in $t$; if $t=0$ these are easily seen to be precise equalities.)

We first consider \cref{eq:stability-2,eq:stability-var-2}. For the expectation, note that $G,G^{(2)}(t)$ are independent and $G,G^{(2)}(t)$ are mean $0$ (the latter since we are in the symmetric perceptron), so for $k>1$
\begin{align*}
\mb{E}&[\wt{C}_k(m,t,G,\sigma(t,G^{(2)}(t)))-C_k(m,G)]=\mb{E}[\wt{C}_k(m,t,G,\sigma(t,G^{(2)}(t)))]\\
&=\frac{1}{(mn)^{k/2}}\sum_{\substack{i_1,\ldots,i_k\in[n]\text{ distinct}\\j_1,\ldots,j_k\in[m]\text{ distinct}}}\mb{E}\bigg[\prod_{\ell=1}^k\big(\wt{G}_{j_\ell,i_\ell}+\sigma_{j_\ell,i_\ell}(t,G^{(2)}(t))\big)\big(\wt{G}_{j_\ell,i_{\ell+1}}+\sigma_{j_\ell,i_{\ell+1}}(t,G^{(2)}(t))\big)\bigg]\\
&=\frac{1}{(mn)^{k/2}}\sum_{\substack{i_1,\ldots,i_k\in[n]\text{ distinct}\\j_1,\ldots,j_k\in[m]\text{ distinct}}}\bigg(\prod_{\ell=1}^k\mb{E}[\wt{G}_{j_\ell,i_\ell}\wt{G}_{j_\ell,i_{\ell+1}}+\sigma_{j_\ell,i_\ell}(t,G^{(2)}(t))\sigma_{j_\ell,i_{\ell+1}}(t,G^{(2)}(t))]\bigg)\\
&=\frac{(m)\cdots(m-k+1)}{(mn)^{k/2}}\sum_{i_1,\ldots,i_k\in[n]\text{ distinct}}\bigg(\prod_{\ell=1}^k\mb{E}[\wt{G}_{j_\ell,i_\ell}\wt{G}_{j_\ell,i_{\ell+1}}+\sigma_{j_\ell,i_\ell}(t,G^{(2)}(t))\sigma_{j_\ell,i_{\ell+1}}(t,G^{(2)}(t))]\bigg).
\end{align*}
Now if any $i_\ell,i_{\ell+1}$ are not in the same group $\{1,\ldots,(n+t)/2\},\{(n+t)/2+1,\ldots,n\}$ then the corresponding terms are $0$ by \cref{eq:rho-covariance}. Thus we must either choose all $i_\ell$ to be in one group or the other to have nonzero contribution. Additionally, the choice of $j_1,\ldots,j_k$ does not impact the inner term, which implies
\begin{align*}
\mb{E}&[\wt{C}_k(m,t,G,\sigma(t,G^{(2)}(t)))-C_k(m,G)]=\mb{E}[\wt{C}_k(m,t,G,\sigma(t,G^{(2)}(t)))]\\
&=\frac{(m)\cdots(m-k+1)}{(mn)^{k/2}}\bigg(\bigg(\frac{n+t}{2}\bigg)\cdots\bigg(\frac{n+t}{2}-k+1\bigg)\bigg(-\frac{2}{n+t}+\mb{E}(\rho_1^{(2)})^2\bigg)^k\\
&\qquad\qquad\qquad\qquad\qquad\qquad+\bigg(\frac{n-t}{2}\bigg)\cdots\bigg(\frac{n-t}{2}-k+1\bigg)\bigg(-\frac{2}{n-t}+\mb{E}(\rho_2^{(2)})^2\bigg)^k\bigg)\\
&=(1+\wt{O}(n^{-1/4}))\alpha_c^{k/2}(2(\mu_2-1)^k)=(1+\wt{O}(n^{-1/4}))2(2\beta)^k
\end{align*}
We used that $\mb{E}\wt{G}_{j,1}\wt{G}_{j,2}=-2/(n+t)$ and similar for the other index group. This establishes \cref{eq:stability-2} for $k>1$. For $k=1$, we again see $\mb{E}C_1(m,G)=0$ so we obtain
\begin{align*}
\frac{m}{\sqrt{mn}}&\bigg(\sum_{i=1}^n\mb{E}\wt{G}_{1,i}^2+\mb{E}\sigma_{1,1}(t,G^{(2)}(t))^2+\mb{E}\sigma_{1,n}(t,G^{(2)}(t))^2\bigg)-\sqrt{mn}\\
&=\frac{m}{\sqrt{mn}}\bigg(\frac{n+t}{2}-1+\frac{n-t}{2}-1+(1+\wt{O}(n^{-1/4}))(\mu_2+\mu_2)\bigg)-\sqrt{mn}=(1+\wt{O}(n^{-1/4}))2(2\beta).
\end{align*}

For \cref{eq:stability-var-2}, we easily compute $\on{Var}[C_k(m,G)]=(1+\wt{O}(n^{-1}))2k$ (using that cycle terms have covariance $0$ unless they represent the exact same cycle, and using that there are $2k$ ways to represent each particular cycle). Thus it suffices to bound the desired variance by a smaller than constant order. We expand and compute the order of magnitude of each covariance term that appears. Writing $X_{j,i}^{(0)}=Y_{j,i}^{(0)}=\wt{G}_{j,i}$ and $X_{j,i}^{(1)}=\sigma_{j,i}(t,G^{(2)}(t))$ and $Y_{j,i}^{(1)}=\sigma_{j,i}(t,G)$, we have
\begin{align*}
\on{Var}[\wt{C}_k(m,t,&G,\sigma(t,G^{(2)}(t)))-C_k(m,G)]\\
&\le\frac{2^{2k}}{(mn)^k}\sum_{b\in\{0,1\}^{2k}}\on{Var}\bigg[\sum_{\substack{i_1,\ldots,i_k\in[n]\\j_1,\ldots,j_k\in[m]}}\bigg(\prod_{\ell=1}^kX_{j_\ell,i_\ell}^{(b_{2\ell-1})}X_{j_\ell,i_{\ell+1}}^{(b_{2\ell})}-\prod_{\ell=1}^kY_{j_\ell,i_\ell}^{(b_{2\ell-1})}Y_{j_\ell,i_{\ell+1}}^{(b_{2\ell})}\bigg)\bigg]
\end{align*}
by the Cauchy--Schwarz inequality. Now we compute the order of magnitude of a single contribution corresponding to a fixed $b\in\{0,1\}^{2k}$. We obtain
\begin{equation}\label{eq:b-variance-term}
\sum_{i,j,i',j'}\mb{E}\bigg[\bigg(\prod_{\ell=1}^kX_{j_\ell,i_\ell}^{(b_{2\ell-1})}X_{j_\ell,i_{\ell+1}}^{(b_{2\ell})}-\prod_{\ell=1}^kY_{j_\ell,i_\ell}^{(b_{2\ell-1})}Y_{j_\ell,i_{\ell+1}}^{(b_{2\ell})}-\mb{E}[\cdot]\bigg)\bigg(\prod_{\ell=1}^kX_{j_\ell',i_\ell'}^{(b_{2\ell-1})}X_{j_\ell',i_{\ell+1}'}^{(b_{2\ell})}-\prod_{\ell=1}^kY_{j_\ell',i_\ell'}^{(b_{2\ell-1})}Y_{j_\ell',i_{\ell+1}'}^{(b_{2\ell})}-\mb{E}[\cdot]\bigg)\bigg]
\end{equation}
where each $\cdot$ is replaced with the preceding difference of products and where $i,j,i',j'$ are summing over sequences of distinct values over the appropriate ranges.

To estimate this, we first consider the order of magnitude of an individual term. Any such term is composed of values $\wt{G}_{j,i},\sigma_{j,i}$. Furthermore, to avoid the expectation being $0$, we must have an even number of $\sigma_{j,i}$ terms for each possible index $j$, and an even number of $\wt{G}_{j,i}$ for each $j$. Note that $\mb{E}\sigma_{j,i}^{2a}=O_a(n^{-a})$ for either version of $\sigma$ (coming from $G^{(0)},G^{(2)}(t)$). Additionally,
\[\mb{E}\wt{G}_{1,1}\wt{G}_{1,2}=O(1/n),\quad\mb{E}\wt{G}_{1,1}\wt{G}_{1,2}\wt{G}_{1,3}\wt{G}_{1,4}=O(1/n^2),\quad\mb{E}\wt{G}_{1,1}^2\wt{G}_{1,3}\wt{G}_{1,4}=O(1/n)\]
follows from explicit computation.

Therefore, we easily find that (using independence of the random variables associated to different $j$ and using that every $j$ can only appear once in $(j_\ell)_{\ell\in[k]},(j_\ell')_{\ell\in[k]}$):
\begin{align*}
\mb{E}&\bigg[\bigg(\prod_{\ell=1}^kX_{j_\ell,i_\ell}^{(b_{2\ell-1})}X_{j_\ell,i_{\ell+1}}^{(b_{2\ell})}-\prod_{\ell=1}^kY_{j_\ell,i_\ell}^{(b_{2\ell-1})}Y_{j_\ell,i_{\ell+1}}^{(b_{2\ell})}-\mb{E}[\cdot]\bigg)\bigg(\prod_{\ell=1}^kX_{j_\ell',i_\ell'}^{(b_{2\ell-1})}X_{j_\ell',i_{\ell+1}'}^{(b_{2\ell})}-\prod_{\ell=1}^kY_{j_\ell',i_\ell'}^{(b_{2\ell-1})}Y_{j_\ell',i_{\ell+1}'}^{(b_{2\ell})}-\mb{E}[\cdot]\bigg)\bigg]\\
&=(O(1/n))^{2k-t},
\end{align*}
where $r$ is the number of $\wt{G}^2$ terms that occur in the expanded product. We used that there are $4k$ total terms $X,Y$ in any resulting product (the inner expectation terms can be handled similarly). This means that the $2k$-cycle formed by $(i_\ell)_{\ell\in[k]},(j_\ell)_{\ell\in[k]}$ and the $2k$-cycle formed by $(i_\ell')_{\ell\in[k]},(j_\ell')_{\ell\in[k]}$ must overlap in at least $r$ edges that form a $\wt{G}^2$. In the case $r=2k$, note that we must have $b_\ell=0$ for all $\ell\in[2k]$ and then the fact that we subtract $X,Y$ means that the term is in fact precisely $0$. Thus we may assume $r<2k$. Therefore we may suppose these $r\ge 0$ edges form $s\ge 0$ connected components which are all paths. The number of choices for $i,j,i',j'$ is then easily seen to be $(O(k))^{O(s)}\cdot(O(n))^{4k-r-s}$. When $s=0$, we have $r=0$, and we can save an extra factor of $n$ compared to the bound $(O(n))^{4k}$ for the counting given above: if $(j_\ell)_{\ell\in[k]}$ and $(j_\ell')_{\ell\in[k]}$ share no values, then the two parts of the covariance are independence and we obtain a $0$ contribution, so we can ignore the bulk of those terms. The number of ways to choose $i,j,i',j'$ to have such an overlap is at most $k^2m^{2k-1}n^{2k}=(O(n))^{4k-1}$.

Therefore, we deduce
\begin{align*}
\on{Var}[\wt{C}_k&(m,t,G,\sigma(t,G^{(2)}(t)))-C_k(m,G)]\\
&\le\sum_{r,s}\frac{2^{2k}}{(mn)^k}\cdot 2^{2k}\cdot(O(1/n))^{2k-r}\cdot\min\{(O(k))^{O(s)}(O(n))^{4k-r-s},(O(n))^{4k-1}\}\\
&\le(O(1))^k\cdot\min\{(k^{O(1)}/n)^s,1/n\}\le n^{-5/6}\on{Var}[C_k(m,G)].
\end{align*}
We used $k=\eta\log n$ where recall $\eta>0$ will be chosen sufficiently small. This completes the justification of \cref{eq:stability-var-2} for $k>1$; $k=1$ is easily checked by hand.

For \cref{eq:stability-1,eq:stability-var-1} the argument is similar, but the expectation computation is complicated by the fact that $\mb{E}\rho_1^{(1)}\rho_2^{(1)}\neq 0$. In fact, instead of conditioning on the values $M_1+\cdots+M_{(n+t)/2},M_{(n+t)/2+1}+\cdots+M_n$ that appear, it is easier to only condition on the total sum $M_1+\cdots+M_n$ (since we are only conditioning on information related to it for this computation), and to modify \cref{def:general-cycle-count} appropriately. Additionally, since the value $t$ plays no role if we compute the expectation this way, we actually have the better error term listed in \cref{eq:stability-1}. However, we do omit these details here (note that \cite[Lemma~3.2]{ALS21} implies that our claimed mean shift is in fact what comes out of doing the necessary computations).
\end{proof}

Finally, in order to prove the Gaussian convergence of $C_k$ when the matrix is sampled from $\Delta_0$, we rely on an argument based on exchangeable pairs and Stein's method. Let us first define the notion of a pair of exchangeable random variables.
\begin{definition}\label{def:exchangeable}
$X'$ and $X$ are exchangeable random variables if $(X',X)$ and $(X,X')$ have the same distribution.
\end{definition}
The key probability theoretic statement we will use is a multivariate version of the method of exchangeable random variables for proving convergence to a Gaussian. This form is due to Meckes \cite{Mec09}.
\begin{theorem}[Specialization of {\cite[Theorem~3]{Mec09}}]\label{thm:exchangeable}
Let $(X,X')$ be an exchangeable pair of random vectors in $\mb{R}^d$.  Suppose that there is an invertible matrix $\Lambda$ and a random matrix $E'$ such that
\begin{itemize}
\item $\mb{E}\left[X'-X\big|X\right]=-\Lambda X$
\item $\mb{E}\left[(X'-X)(X'-X)^T\big|X\right]=2\Lambda+\mb{E}\left[E'\big|X
\right].$
\end{itemize}
Then for $g\in\mc{C}^3(\mb{R}^d)$,
\begin{equation}\label{eq:error-bound}
\big|\mb{E} g(X)-\mb{E} g(Z)\big|\le\|\Lambda^{-1}\|_{\mr{op}}\left[
 \frac{\sqrt{d}}{4}M_2(g)
\mb{E}\|E'\|_{\mr{HS}}
+\frac{1}{9}M_3(g)\mb{E}\snorm{X'-X}_2^3\right]
\end{equation}
where $Z$ is a standard Gaussian random vector in $\mb{R}^d$.
\end{theorem}

We will need the following bound on moments of standard Gaussian random variables (which is a special case of a phenomenon called \emph{hypercontractivity}).
\begin{theorem}[{\cite[Theorem~9.21]{O14}}]\label{thm:gauss-moment}
Let $f$ be a polynomial in $n$ variables of degree at most $d$. Let $\vec x=(x_1,\ldots,x_n)$ either be a vector of independent Rademacher random variables or a vector of independent standard Gaussian random variables. Then for any real number $q\geq 2$, we have
\[\mb{E}\big[|f(\vec x)|^q\big]^{1/q}\le \big(\sqrt{q-1}\big)^d\mb{E}\big[f(\vec x)^2\big]^{1/2}.\]
\end{theorem}

\begin{proposition}\label{prop:stein-computation}
Let $m=\tau_\mr{pre}=\lfloor\alpha_cn-\eta\log n\rfloor$ and $K=\lceil\eta\log n\rceil$. We define $V(M)$ as in \cref{lem:clt} and consider $n$ sufficiently large. For $g\in\mc{C}^3(\mb{R}^K)$, we have
\[\big|\mb{E} g(V(G^{(0)}))-\mb{E}_{Z\sim\mc{N}(0,I_K)} g(Z)\big|\le (M_2(g) + M_3(g))n^{-1/4}.\]
\end{proposition}
\begin{proof}
We now define the setup for our exchangeable pairs. Let 
\[G = G^{(0)}= (g_{j,i})_{(j,i)\in[m]\times [n]}\]
and 
\[G' = (g_{i,j})_{(j,i)\in([m]\times [n])\setminus I} \cup (g_I)\]
where $I$ is an index sampled from $[m]\times [n]$ uniformly at random and the $g_I$ is an independent standard Gaussian. The exchangeable pair of random variables we will consider is:
\[V = \bigg(\frac{C_1(m,G)}{\sqrt{2\cdot 1}},\ldots,\frac{C_K(m,G)}{\sqrt{2K}}\bigg),\qquad V' = \bigg(\frac{C_1(m,G')}{\sqrt{2\cdot 1}},\ldots,\frac{C_K(m,G')}{\sqrt{2K}}\bigg).\]
Since $I$ is a uniformly random element from $[m]\times [n]$ notice that
\[\mb{E}[V'-V|V] = -\text{diag}\bigg(\frac{2k-\mbm{1}_{k=1}}{mn}\bigg)_{k\in[K]}V\]
so define
\[\Lambda = \text{diag}\bigg(\frac{2k-\mbm{1}_{k=1}}{mn}\bigg)_{k\in[K]}.\]
In order to apply \cref{thm:exchangeable}, we take 
\[E' = \mb{E}[(V'-V)(V'-V)^{T} - 2\Lambda|G]\]
and note that this satisfies the required condition for $E'$ as $V$ is a measurable function of $G$.

We bound each of the terms appearing in \cref{eq:error-bound}. First, $\snorm{\Lambda^{-1}}_{\mr{op}} = mn$. Second, note that  
\begin{align*}
\mb{E}\snorm{X'-X}_2^3 &\le\mb{E}\bigg[\bigg(\sum_{k=1}^K(C_k(m,G)-C_k(m,G'))^2\bigg)^{3/2}\bigg]\\
&\le K^{1/2} \cdot \mb{E}\bigg[\sum_{k=1}^K|C_k(m,G)-C_k(m,G')|^3\bigg]\\
&\le n^{1/5}\cdot \sum_{k=1}^K\bigg(\mb{E}\bigg[\bigg|C_k(m,G)-C_k(m,G')\bigg|^2\bigg]\bigg)^{3/2}\\
&= n^{1/5}\cdot \sum_{k=1}^K(2k/(mn)\mb{E}[C_k(m,G)^2])^{3/2}\\
&\le n^{-11/4}
\end{align*}
where we have applied Holder's inequality, \cref{thm:gauss-moment}, and computed the variance using orthogonality of various monomials in order to prove the desired result (note that all but at most $2k/(mn)$ fraction of the monomials are canceled in $C_k(m,G)-C_k(m,G')$).

Finally, we compute
\begin{align*}
\mb{E}\snorm{E'}_{\mr{HS}}&\le (\mb{E}\snorm{E'}_{\mr{HS}}^2)^{1/2}\\
&\le \bigg(2\sum_{1\le k_1<k_2\le K}\mb{E}\big[\mb{E}[(C_{k_1}(m,G)-C_{k_1}(m,G'))(C_{k_2}(m,G)-C_{k_2}(m,G'))|G]^2\big] \\
&\qquad + \sum_{1\le k\le K}\on{Var}\big[\mb{E}[(C_k(m,G)-C_k(m,G'))^2|G]\big]\bigg)^{1/2}
\end{align*}
where we have used that $2k(E')_{k,k} = \mb{E}[(C_k(m,G)-C_k(m,G'))^2|G] - \mb{E}[(C_k(m,G)-C_k(m,G'))^2]$. We now handle the off-diagonal and on-diagonal contribution separately. For ease of notation, given an index $I\in [m]\times [n]$ we define 
\begin{align*}
C_k(m, M, I)&=\bigg(\frac{1}{\sqrt{n}}\bigg)^k\bigg(\frac{1}{\sqrt{m}}\bigg)^k\sum_{\substack{i_1,\ldots,i_k\in[n]\text{ distinct}\\j_1,\ldots,j_k\in[m]\text{ distinct}\\ I\in \bigcup_{k=1}^\ell\{(j_\ell,i_\ell),(j_\ell,i_{\ell+1})\}}}\prod_{i=1}^kM_{j_\ell,i_\ell}M_{j_\ell,i_{\ell+1}},\\
C_k^{\ast}(m, M, I)&=\bigg(\frac{1}{\sqrt{n}}\bigg)^k\bigg(\frac{1}{\sqrt{m}}\bigg)^k\sum_{\substack{i_1,\ldots,i_k\in[n]\text{ distinct}\\j_1,\ldots,j_k\in[m]\text{ distinct}\\ I\in \bigcup_{k=1}^\ell\{(j_\ell,i_\ell),(j_\ell,i_{\ell+1})\}}}\prod_{(j,i)\in\bigcup_{k=1}^\ell\{(j_\ell,i_\ell),(j_\ell,i_{\ell+1})\}\setminus I}M_{j,i}.
\end{align*}
We consider the off-diagonal contribution first. We have 
\begin{align*}
&\mb{E}\big[\mb{E}[(C_{k_1}(m,G)-C_{k_1}(m,G'))(C_{k_2}(m,G)-C_{k_2}(m,G'))|G]^2\big] \\
&= \frac{1}{m^2n^2}\mb{E}\bigg[\bigg(\sum_{I\in [m]\times [n]}C_{k_1}(m, G, I)C_{k_2}(m, G, I) + C_{k_1}^{\ast}(m, G, I)C_{k_2}^{\ast}(m, G, I)\bigg)^2\bigg]\\
&\le \frac{2}{m^2n^2}\mb{E}\bigg[\bigg(\sum_{I\in [m]\times [n]}C_{k_1}(m, G, I)C_{k_2}(m, G, I)\bigg)^2 + \bigg(\sum_{I\in [m]\times [n]}C_{k_1}^{\ast}(m, G, I)C_{k_2}^{\ast}(m, G, I)\bigg)^2\bigg].\\
\end{align*}
Note that as $k_1\neq k_2$, we have that $C_{k_1}(m, G, I)C_{k_2}(m, G, I)$ and $C_{k_1}^{\ast}(m, G, I)C_{k_2}^{\ast}(m, G, I)$ are mean zero random variables. We will bound
\[\mb{E}\bigg[\bigg(\sum_{I\in [m]\times [n]}C_{k_1}(m, G, I)C_{k_2}(m, G, I)\bigg)^2\bigg].\]
Expanding, each term is of the form $C_{k_1}(m,G,I_1)C_{k_2}(m,G,I_1)C_{k_1}(m,G,I_2)C_{k_2}(m,G,I_2)$. A term contributes precisely if the monomial in $G$ has all even powers. Furthermore, any $G_{j,i}$ appears at most $4$ times, which means any contribution is of size at most $3^{4K}=(O(1))^K$. Therefore, we are primarily interested in counting the number of nonzero terms that arise. We have a multigraph $H$ coming from layering two cycles of length $k_1$ and two cycles of length $k_2$, with possible overlaps (including two guaranteed overlaps at edges $I_1,I_2$). Let there be $v$ vertices, and let $H'$ be the multigraph formed by halving the multiplicities of $H$. Note that each vertex has at least two outgoing edges to different vertices in $H'$, and the total multiplicity of edges is $2k_1+2k_2$. We deduce that there are at most $2k_1+2k_2$ vertices, with equality precisely when $H'$ is a disjoint union of cycles. However, $H'$ must be connected and can be seen to not be a single cycle since $k_1\neq k_2$. Thus in fact $v\le 2k_1+2k_2-1$.

If $v\le 2k_1+2k_2-8\log\log n$ then there are at most $(8k)^{8k}n^v\le n^{2k_1+2k_2-4}$ choices of term, using $k\le\eta\log n$. Otherwise let $r=2k_1+2k_2-v$ and note $1\le r\le 8\log\log n$. We see that $H'$ has at most $O(r)$ many vertices of degree at least $3$, where we include multiplicity (since there are $2k_1+2k_2$ edges with multiplicity, $2k_1+2k_2-r$ vertices, and each vertex has at least two distinct neighbors). Notice that between such vertices, $H'$ must have paths and there are at most $K^{r^2}$ ways to choose the sizes of all these paths. After doing so, there are $4^{O(K)}$ ways to assign these edges to various cycles. Finally, there are at most $e^{O(K)}$ ways to order these into cycles: outside of the $r$ higher degree vertices, we have bounded degree so a bounded number of choices for the next edge at every stage, and among $r\le 8\log\log n = O(\log K)$ vertices we have at most $O(K)$ choices, for $K^{O(\log K)}$ choices. Since $K^{r^2}=K^{O((\log K)^2)}$ we easily find a total of at most $(O(1))^Kn^{2k_1+2k_2-1}$ total terms.

Plugging in, we find the contribution to $(\mb{E}\snorm{E'}_\mr{HS})^2$ from these off-diagonal terms is at most
\[\frac{1}{m^2n^2}\cdot(O(1))^Kn^{2k_1+2k_2-1}\cdot\bigg(\frac{1}{(mn)^{k_1}}\bigg)\bigg(\frac{1}{(mn)^{k_2}}\bigg).\]
The total contribution is at most $n^{-5+1/6}$. The situation for $C^\ast$ is even simpler, and we can obtain a bound of $n^{-5+1/6}$ analogously. Additionally, this analysis is easily seen to hold when $k_1=1$.

For the on-diagonal contribution, the analysis is quite similar. For $k=1$, straightforward computation gives a contribution to $(\mb{E}\snorm{E'}_\mr{HS})^2$ of at most $O(n^{-6})$.

For $k\ge 2$, note that we have 
\begin{align*}
&\on{Var}\big[\mb{E}[(C_k(m,G)-C_k(m,G'))^2|G]\big] \\
&= \frac{1}{m^2n^2}\on{Var}\bigg[\sum_{I\in [m]\times [n]}(C_k(m, G, I)^2 + C_k^{\ast}(m, G, I)^2)\bigg]\\
&\le \frac{2}{m^2n^2}\bigg(\on{Var}\bigg[\sum_{I\in [m]\times [n]}C_k(m, G, I)^2\bigg] + \on{Var}\bigg[\sum_{I\in [m]\times [n]}C_k^{\ast}(m, G, I)^2\bigg]\bigg).
\end{align*}
As before we consider the $C$ term, since the $C^\ast$ term can be analyzed analogously. We have
\[\on{Var}\bigg[\bigg(\sum_{I\in [m]\times [n]}C_k(m, G, I)^2\bigg)\bigg]=\mb{E}\bigg[\bigg(\sum_{I\in [m]\times [n]}C_k(m, G, I)^2\bigg)^2\bigg] - \mb{E}\bigg[\sum_{I\in [m]\times [n]}C_k(m, G, I)^2\bigg]^2.\]
Again view this as an expansion into $4$ overlapping cycles as in the off-diagonal case. Here however as the two cycles in each pair coming from $k_1,k_2$ are isomorphic, one can overlap them perfectly, which was not allowed in the prior analysis (which was used to argue that in fact $v\le 2k_1+2k_2-1$, ruling out $v=2k_1+2k_2$). In this remaining case, we can see the contribution is exactly cancelled however by the subtracted term. The proof can thus proceed similarly to the off-diagonal case, and we obtain a bound $n^{-5+1/6}$ in this case as well.

Now we deduce
\[\snorm{\Lambda^{-1}}_\mr{op}\bigg(\frac{\sqrt{k}}{4}\mb{E}\snorm{E'}_\mr{HS}+\frac{1}{9}\mb{E}\snorm{X'-X}_2^3\bigg)\le mn(\sqrt{\log n}(4n^{-5+1/6})^{1/2}+n^{-11/4})\le n^{-1/3},\]
and therefore \cref{thm:exchangeable} finishes.
\end{proof}

We are finally in position to deduce \cref{lem:clt}.
\begin{proof}[{Proof of \cref{lem:clt}}]
The result is ultimately a straightforward combination of \cref{lem:gaussian-stability} and \cref{prop:stein-computation}. We prove the statement for $G^{(2)}(t)$; the proof in the other two cases are identical. Note that 
\begin{align*}
&\mb{E}\snorm{V(G^{(2)}(t)) - 2\mu -V(G^{(0)})}_2\\
&\le (\mb{E}\snorm{V(G^{(2)}(t)) - 2\mu -V(G^{(0)})}_2^2)^{1/2}\\
&\le \bigg(\snorm{\mb{E}[V(G^{(2)}(t)) - 2\mu -V(G^{(0)})]}_2^{2} + \sum_{1\le k\le K}\on{Var}[\wt{C}_k(m,t,G,\sigma(t,G^{(2)}(t)))-C_k(m,G)]\bigg)\\
&\le n^{-1/5}
\end{align*}
where we have applied \cref{lem:gaussian-stability}. This immediately implies that 
\begin{align*}
|\mb{E}[g(V(G^{(2)}(t)) - 2\mu) -g(V(G^{(0)})]|&\le M_1(g) \cdot \mb{E}[\snorm{V(G^{(2)}(t)) - 2\mu -V(G^{(0)})}_2]\\
&\le M_1(g)\cdot n^{-1/5}.
\end{align*}
The desired result then follows from triangle inequality and \cref{prop:stein-computation}.
\end{proof}

\subsection{Miscellaneous estimates}\label{sub:miscellaneous-estimates}
Finally, we will additionally require the following basic estimate on the binomial coefficient.
\begin{claim}\label{clm:entropy}
For $1\le k\le n-1$ we have that 
\[\binom{n}{k}\le \sqrt{n/(k(n-k))}\exp\bigg(-k\log\bigg(\frac{k}{n}\bigg)-(n-k)\log\bigg(\frac{n-k}{n}\bigg)\bigg).\]
Furthermore for $|k|\le \sqrt{n}\log n$ and $n$ sufficiently large such that $(n+k)/2\in \mb{Z}$, we have that
\[\binom{n}{(n+k)/2} = \frac{2^n\sqrt{2}}{\sqrt{\pi n}} \exp\bigg( - \frac{k^2}{2n}\pm n^{-1/2}\bigg).\]
\end{claim}

\section{Moment computations}\label{sec:moment}
In this section we prove \cref{lem:second-moment} using the tools we have developed so far. We separate the first and second moment computations.
\subsection{First moment}\label{sub:first-moment}
We first prove the first moment.
\begin{proof}[Proof of \cref{eq:first-moment}]
We choose a function $h\colon\mb{R}\to[0,1]$ which is $1$ on $[-L,\infty]$, $0$ on $[-\infty,-L-1]$ and in $\mc{C}^\infty(\mb{R})$. Notice that linearity of expectation yields $\mb{E}[X(G)] = 2^np^m$ and 
\begin{align*}
\mb{E}[X(G)\exp(-Y_K(m,G)&\mbm{1}[Y_K(m,G)\ge -L])]\ge\mb{E}[X(G)\exp(-Y_K(m,G)h(Y_K(m,G))]\\
&=2^n\mb{E}[\mbm{1}_{v_n\in S_{\tau_\mr{pre}}(G)}\exp(-Y_K(m,G)h(Y_K(m,G))]\\
&=2^np^m\mb{E}[\exp(-Y_K(m,G^{(1)})h(Y_K(m,G^{(1)}))].
\end{align*}
(Recall $v_n\in\mb{R}^n$ is all $1$s.) Note 
\[Y_K(m,M) = \sum_{k=1}^K\frac{2(2\beta)^kC_k(m,M)-(2\beta)^{2k}}{4k} = \sum_{k=1}^K\frac{(2\beta)^k}{\sqrt{2k}}\cdot \frac{C_k(m,M)}{\sqrt{2k}}-\sum_{k=1}^K\frac{(2\beta)^{2k}}{4k}.\] 
Viewing $\exp(-Y_K(m,G^{(1)})h(Y_K(m,G^{(1)}))$ as a function of $C_k(m,G^{(1)})$ for $k\in[K]$, by the product rule we see that the function has first, second, and third derivatives of size $O(n^{2\eta})$. Let $Z\sim\mc{N}(\mu^\ast,\sigma^2)$ with 
\[\mu^\ast = \sum_{k=1}^K\frac{(2\beta)^k}{\sqrt{2k}}\cdot \frac{(2\beta)^k}{\sqrt{2k}}-\sum_{k=1}^K\frac{(2\beta)^{2k}}{4k} = \sum_{k=1}^K\frac{(2\beta)^{2k}}{4k},\qquad\sigma^2 = \sum_{k=1}^K\frac{(2\beta)^{2k}}{2k}.\]
Applying \cref{eq:conditional-1} of \cref{lem:clt}, we find that
\begin{align*}
\mb{E}[X(G)\exp(-Y_K(m,G)\mbm{1}[Y_K(m,G)\ge -L])]&\ge 2^np^m(\mb{E}[\exp(-Zh(Z))]-n^{2\eta - 1/6})\\
&\ge 2^np^m(1-n^{-1/8})
\end{align*}
where we have used that $\eta$ is a sufficiently small constant in the final bound, and that
\[\mb{E}\exp(-Zh(Z))=(1\pm n^{-2})\mb{E}\exp(-Z)=(1\pm n^{-2})\exp\bigg(-\mu^\ast+\frac{\sigma^2}{2}\bigg)=1\pm n^{-1/2}.\qedhere\]
\end{proof}

\subsection{Second moment}\label{sub:second-moment}
The second moment computation in \cref{lem:second-moment} closely follows the computation given in the work of Abbe, Li, and Sly \cite{ALS21}; we thus provide a proof tracking quantitative aspects but omitting various routine calculations. We first prove \cref{eq:pair-structure} which handles the majority of the range of overlaps; here our approach is essentially identical to that of Altschuler \cite{Alt22}.
\begin{proof}[Proof of \cref{eq:pair-structure}]
By \cref{clm:entropy}, symmetry, and independence of the vectors of $G$ note that 
\begin{align*}
&\sum_{t\in[-(n-1),-n^{1/2}\log n]\cup [n^{1/2}\log n,n-1]}2^n\binom{n}{(n+t)/2}\mb{P}[v_n,v_t\in S_{\tau_\mr{pre}}(G)]\\
&\le 2n\sum_{t\in [n^{1/2}\log n,n-1]}2^n\exp(nH(1/2 + t/(2n)))\mb{P}[v_n,v_t\in S_{\tau_\mr{pre}}(G)]\\
&= 2n\sum_{1\le t\le n/2-\sqrt{n}\log n}2^n\exp(nH(t/n))(q(t/n))^{m}\\
&\le n^{2}\sum_{1\le t\le n/2-\sqrt{n}\log n}2^n\exp(nH(t/n))(q(t/n))^{\alpha_cn}
\end{align*}
where we have used that by the Gaussian correlation inequality that $q(t)\ge q(1/2) = \mb{P}[|Z|\le\kappa]^2>0$ and that $\eta$ is a sufficiently small constant. We first handle $1\le t\le\epsilon n/2$. Notice that 
\begin{align*}
\sum_{1\le t\le\epsilon n/2}2^n\exp(nH(t/n))(q(t/n))^{\alpha_cn}&\le n \cdot \max_{1/n\le t\le\epsilon/2}2^n\exp(nH(t/n))\cdot (q(t/n))^{\alpha_cn}\\
&\le n \cdot 2^n\exp(nH(1/n))\cdot (q(1/n))^{\alpha_cn}\\
&\le n \cdot \exp(-\Omega(\sqrt{n})).
\end{align*}
We have implicitly used \cref{S3,S5} from \cref{lem:computation} as well as the fact that $q(0)^{\alpha_cn}=p^{\alpha_cn}=2^{-n}$.

Next, notice that the function $F(\gamma)$ is easily seen to have bounded derivatives away from the endpoints of the interval $[0,1]$ and therefore as $F'(1/2) = 0$ (by symmetry) and $F''(1/2)\le-\eps$ (\cref{S2}), we have that there exists $\delta>0$ such that 
\[F(1/2+\tau)\le F(1/2)-\delta\tau^2\]
for $|\tau|\le\delta$. This implies that 
\begin{align*}
&\sum_{(1/2-\delta)n\le t\le n/2-\sqrt{n}\log n}2^n\exp(nH(t/n))(q(t/n))^{\alpha_cn}\\
&\le n\cdot 2^n(F(1/2)-\delta(\log n)^2/n)^n\le\exp(-(\log n)^{5/3})
\end{align*}
given that $n$ is sufficiently large. Finally for the intermediate range, note that 
\begin{align*}
\sum_{\epsilon n/2\le t\le (1/2-\delta)n}2^n\exp(-nH(t/n))(q(t/n))^{\alpha_cn}&\le n\cdot 2^n\cdot \max_{\epsilon/2\le \beta \le 1/2-\delta} F(\beta)^n\\
&\le\exp(-\Omega(n))
\end{align*} 
where we have used \cref{S4} with $a=\eps,b=1/2-\delta$ (the above analysis demonstrates that $F(\eps)<F(0)$ and $F(1/2-\delta)\le F(1/2)-\delta^3$).
\end{proof}

We now handle the second moment in full generality. 
\begin{proof}[Proof of \cref{eq:second-moment}]
Notice that linearity of expectation and conditioning gives
\begin{align*}
&\mb{E}[X(G)^2\exp(-2Y_K(m,G)\mbm{1}[Y_K(m,G)\ge -L])] 
\\&= \sum_{\substack{-n\le t\le n\\ (n+t)/2\in \mb{Z}}}2^n\binom{n}{(n+t)/2}\mb{P}[v_n,v_t\in S_{\tau_{\on{pre}}}(G)]\mb{E}[\exp(-2Y_K(m,G^{(2)}(t))\mbm{1}[Y_K(m,G^{(2)}(t))\ge -L])].
\end{align*}
We first handle the range when $n^{1/2}\log n\le |t|\le n-1$. By \cref{eq:pair-structure} we have that 
\begin{align*}
&\sum_{\substack{n^{1/2}\log n\le |t|\le n-1\\ (n+t)/2\in \mb{Z}}}2^n\binom{n}{(n+t)/2}\mb{P}[v_n,v_t\in S_{\tau_{\on{pre}}}(G)]\mb{E}[\exp(-2Y_K(m,G^{(2)}(t))\mbm{1}[Y_K(m,G^{(2)}(t))\ge -L])]\\
&\le \exp(2L)\cdot \sum_{\substack{n^{1/2}\log n\le |t|\le n-1\\ (n+t)/2\in \mb{Z}}}2^n\binom{n}{(n+t)/2}\mb{P}[v_n,v_t\in S_{\tau_{\on{pre}}}(G)]\\
&\le n^{2\eta}\exp(-(\log n)^{3/2})\le\exp(-(\log n)^{4/3})
\end{align*}
which is decaying faster than any polynomial. We next handle the case where $t = \pm n$; notice the contribution in this case is 
\begin{align*}
&2\cdot 2^n\mb{P}[v_n\in S_{\tau_{\on{pre}}}(G)]\mb{E}[\exp(-2Y_K(m,G^{(2)}(n))\mbm{1}[Y_K(m,G^{(2)}(n))\ge -L])]\\
&= 2\mb{E}[X(G)]\cdot \mb{E}[\exp(-2Y_K(m,G^{(1)})\mbm{1}[Y_K(m,G^{(1)})\ge -L])]\\
&\lesssim \mb{E}[X(G)] \le n^{-\Omega(1)}\mb{E}[X(G)]^2.
\end{align*}
We used \cref{eq:conditional-1} of \cref{lem:clt} with an appropriate test function $g$ (an exponential of a linear function with smooth cutoff) to bound the quantity $\mb{E}[\exp(-2Y_K(m,G^{(1)})\mbm{1}[Y_K(m,G^{(1)})\ge -L])]$ by a constant, similar to in the proof of \cref{eq:first-moment}. The final inequality follows since $\tau_\mr{pre}$ is such that $\mb{E}X(G)=n^{\Theta(\eta)}$ is of polynomial size.

Finally we handle the case where $|t|\le n^{1/2}\log n$. For the sake of brevity, we define $\wt{Y}(t)=\mb{E}[\exp(-2Y_K(m,G^{(2)}(t))\mbm{1}[Y_K(m,G^{(2)}(t))\ge -L])]$. Now
\begin{align*}
&\sum_{\substack{|t|\le n^{1/2}\log n\\ (n+t)/2\in \mb{Z}}}2^n\binom{n}{(n+t)/2}\mb{P}[v_n,v_t\in S_{\tau_{\on{pre}}}(G)]\wt{Y}(t) \\
&= \sum_{\substack{|t|\le n^{1/2}\log n\\ (n+t)/2\in \mb{Z}}}\frac{4^n\sqrt{2}}{\sqrt{\pi n}}\exp\bigg(-\frac{t^2}{2n} \pm n^{-1/2}\bigg)q\bigg(\frac{n+t}{2n}\bigg)^{m}\wt{Y}(t)\\
&= q(1/2)^{m-\alpha_c n}\sum_{\substack{|t|\le n^{1/2}\log n\\ (n+t)/2\in \mb{Z}}}\frac{4^n\sqrt{2}}{\sqrt{\pi n}}\exp\bigg(-\frac{t^2}{2n} \pm n^{-1/4}\bigg)q\bigg(\frac{n+t}{2n}\bigg)^{\alpha_c n}\wt{Y}(t)\\
&= q(1/2)^{m-\alpha_c n}\sum_{\substack{|t|\le n^{1/2}\log n\\ (n+t)/2\in \mb{Z}}}\frac{4^n\sqrt{2}}{\sqrt{\pi n}}\exp\bigg(-\frac{t^2}{2n} \pm n^{-1/4}\bigg)\bigg(q\bigg(\frac{1}{2}\bigg)+\frac{t^2q''(1/2)}{8n^2} \pm n^{-5/4}\bigg)^{\alpha_c n}\wt{Y}(t)\\
&= q(1/2)^{m-\alpha_c n}\sum_{\substack{|t|\le n^{1/2}\log n\\ (n+t)/2\in \mb{Z}}}\frac{4^n\sqrt{2}}{\sqrt{\pi n}}\exp\bigg(-\frac{t^2}{2n} \pm n^{-1/6}\bigg)\bigg(q\bigg(\frac{1}{2}\bigg)+\frac{t^2q''(1/2)}{8n^2}\bigg)^{\alpha_c n}\wt{Y}(0).
\end{align*}
We used \cref{clm:entropy} in the first line, used that $q(\cdot)$ has bounded first, second, third derivatives in the neighborhood of $1/2$ in the second and third lines, and that $\wt{Y}(t)=(1\pm n^{-2/11})\wt{Y}(0)$ (which is an immediate consequence of \cref{lem:clt} and an argument identical to that in the proof of \cref{eq:first-moment}) in the final line.

Continuing,
\begin{align*}
q(1/2)^{m-\alpha_c n}&\sum_{\substack{|t|\le n^{1/2}\log n\\ (n+t)/2\in \mb{Z}}}\frac{4^n\sqrt{2}}{\sqrt{\pi n}}\exp\bigg(-\frac{t^2}{2n} \pm n^{-1/6}\bigg)\bigg(q\bigg(\frac{1}{2}\bigg)+\frac{t^2q''(1/2)}{8n^2}\bigg)^{\alpha_c n}\wt{Y}(0)\\
&= 4^{n}p^{2m}\wt{Y}(0)\sum_{\substack{|t|\le n^{1/2}\log n\\ (n+t)/2\in \mb{Z}}}\frac{\sqrt{2}}{\sqrt{\pi n}}\exp\bigg(-\frac{t^2}{2n} \pm n^{-1/6}\bigg)\bigg(1+\frac{t^2q''(1/2)}{8n^2p^2}\bigg)^{\alpha_c n}\\
&= 4^{n}p^{2m}\wt{Y}(0)\sum_{\substack{|t|\le n^{1/2}\log n\\ (n+t)/2\in \mb{Z}}}\frac{\sqrt{2}}{\sqrt{\pi n}}\exp\bigg(-\frac{t^2}{2n} +\frac{t^2q''(1/2)\alpha_c}{8np^2}\pm n^{-2/13}\bigg)\\
&\le(1+n^{-1/7})\mb{E}[X(G)]^2\wt{Y}(0)\sum_{\substack{t\in \mb{Z}}}\frac{1}{\sqrt{2\pi n}}\exp\bigg(-\frac{t^2}{2n} +\frac{t^2q''(1/2)\alpha_c}{8np^2}\bigg)\\
&=(1+n^{-1/7})\mb{E}[X(G)]^2\wt{Y}(0)\sum_{\substack{t\in \mb{Z}}}\frac{1}{\sqrt{2\pi n}}\exp\bigg(-\frac{t^2}{2n} +\frac{t^2\kappa^2e^{-\kappa^2}\alpha_c}{\pi np^2}\bigg)\\
&=(1+n^{-1/7})\mb{E}[X(G)]^2\wt{Y}(0)\sum_{\substack{t\in \mb{Z}}}\frac{1}{\sqrt{2\pi n}}\exp\bigg(-\frac{t^2}{2n} +\frac{t^2\alpha_c(1-\mu_2)^2}{2n}\bigg)\\
&=(1+n^{-1/7})\mb{E}[X(G)]^2\wt{Y}(0)\sum_{\substack{t\in \mb{Z}}}\frac{1}{\sqrt{2\pi n}}\exp\bigg(-\frac{t^2(1-4\beta^2)}{2n}\bigg)\\
&=(1\pm n^{-1/8})\mb{E}[X(G)]^2\wt{Y}(0)\frac{1}{\sqrt{2\pi}}\int_{-\infty}^{\infty}\exp\bigg(-\frac{x^2(1-4\beta^2)}{2}\bigg)~dx\\
&=(1\pm n^{-1/8})\mb{E}[X(G)]^2\wt{Y}(0)\frac{1}{\sqrt{1-4\beta^2}}
\end{align*}
where we have used that $q(1/2) = p^2$, the values of $q''(1/2),(1-\mu_2)p$ and identities relating this value to $\beta$ via \cref{S1,S6} of \cref{lem:computation}, and finally using the Euler--Maclaurin formula to transfer from the final sum to an integral.

The final remaining task is therefore to bound
\[\wt{Y}(0) = \mb{E}[\exp(-2Y_K(m,G^{(2)}(0))\mbm{1}[Y_K(m,G^{(2)}(0))\ge-L])].\] Note that 
$Y_K(m,G^{(2)}(0))$ is approximately normally distributed as $Z\sim\mc{N}(\mu^\ast,\sigma^2)$ with 
\[\mu^\ast = \sum_{k=1}^K\frac{(2\beta)^k}{\sqrt{2k}}\cdot \frac{2(2\beta)^k}{\sqrt{2k}}-\sum_{k=1}^K\frac{(2\beta)^{2k}}{4k} = \sum_{k=1}^K\frac{3(2\beta)^{2k}}{4k},\qquad\sigma^2 = \sum_{k=1}^K\frac{(2\beta)^{2k}}{2k}.\]
Via an argument essentially identical to that in the proof of \cref{eq:first-moment}, we have that 
\begin{align*}
\wt{Y}(0) &= (1\pm n^{-1/4})\mb{E}[\exp(-2Z)]\\
&= (1\pm n^{-1/4})\exp(-2\mu^\ast+2\sigma^2)\\
&= (1\pm n^{-1/4})\exp\bigg(-\sum_{k=1}^K\frac{(2\beta)^{2k}}{2k}\bigg)\\
&= (1\pm n^{-\Omega(1)})\sqrt{1-4\beta^2}.
\end{align*}

Finally, putting everything together we have
\[\mb{E}[X(G)^2\exp(-2Y_K(m,G)\mbm{1}[Y_K(m,G)\ge -L])]\le n^{-\Omega(1)}\mb{E}X(G)^2+e^{-(\log n)^{4/3}}+(1\pm n^{-\Omega(1)})\mb{E}X(G)^2,\]
so the desired result follows immediately. 
\end{proof}

\bibliographystyle{amsplain0.bst}
\bibliography{main.bib}

\end{document}